\documentclass[reqno,11pt]{amsart}
\usepackage{amssymb}
\usepackage{amsmath}
\usepackage{amsfonts}
\usepackage{hyperref}

\newtheorem{theorem}{Theorem}
\theoremstyle{plain}

\newtheorem{corollary}{Corollary}

\newtheorem{definition}{Definition}
\newtheorem{example}{Example}

\newtheorem{lemma}{Lemma}

\newtheorem{proposition}{Proposition}
\newtheorem{remark}{Remark}

\numberwithin{equation}{section}

\begin{document}
\title[]{Petrov type extension for multivalued contraction mappings}
\author{Hakan Sahin}
\address{Department of Mathematics, Faculty of Engineering and Natural
Sciences, Bursa Technical University, 16310 Y\i ld\i r\i m, Bursa, Turkey}
\email{hakan.sahin@btu.edu.tr}
\author{Mustafa Aslantas}
\address{Department of Mathematics, Faculty of Science, \c{C}ank\i r\i \
Karatekin University, 18100, \c{C}ank\i r\i , Turkey}
\email{maslantas@karatekin.edu.tr}
\author{Ishak Altun}
\address{Department of Mathematics, Faculty of Engineering and Natural
Sciences, K\i r\i kkale University, 71450, Yahsihan, K\i r\i kkale, Turkey}
\email{ishakaltun@yahoo.com}
\subjclass[2000]{Primary 47H10, Secondary 54H25, 47H04}
\keywords{Multivalued contraction, fixed point theorem, multivalued
contracting perimeters of triangles}

\begin{abstract}
In this paper, we introduce the concept of multivalued $\lambda $%
-contracting perimeters of triangles. This concept generalizes the Nadler's
contraction by considering triplets of points instead of pairs. Fundamental
properties of such mappings are analyzed, including their continuity and
their relationship to classical multivalued contractions. By means of a
counterexample, we show that a mapping which satisfies the condition of
multivalued $\lambda $-contracting perimeters of triangles is not
necessarily a multivalued $\lambda $-contraction. We also introduce the
notion of property of forming a triangle. Then, we investigate the relation
between this property and the fact that there is no periodic point of prime
period $2$ for any multivalued mapping. Furthermore, using the new concept
and property mentioned above, we present some fixed point results for
multivalued mappings. Finally, we provide an illustrative and comparative
example.
\end{abstract}

\maketitle

\renewcommand{\thefootnote}{} \footnotetext{$^*$Corresponding author}

\section{Introduction and Preliminaries}

Fixed point theory has a crucial role in many areas of mathematics and has
extensive applications in various disciplines. The Banach contraction
principle is considered to be the origin of fixed point theory in metric
spaces \cite{Banach}. Established by Banach in 1922, it guarantees the
existence and uniqueness of a fixed point for a contraction mapping in a
complete metric space. Its versatility and wide applicability make it an
important result in fixed point theory. Hence, there are many
generalizations of this result in the literature. These generalizations can
be broadly divided into two main classes. First class involves relaxing the
contraction condition. For example, the classical contraction condition,
which requires that the distance between points is reduced by a constant
factor, has been weakened in several directions. Generalized contraction
mappings, such as those that allow for certain types of nonlinearities or
weaker contraction properties, have been studied \cite{B5, B3, B6, B2, B4,
B1}. The second class consists of generalizations obtained by weakening the
topology. In particular, researchers have explored the application of
Banach's result in more general spaces, including nonmetric spaces,
topological spaces, and metric-like structures \cite{T3, T4, T2, T1}. On the
other hand, the extension of the Banach fixed point theorem to multivalued
mappings has been an important area of research. Notably, Nadler's fixed
point theorem generalizes Banach's result to multivalued maps in complete
metric spaces \cite{Nadler} as follows.

\begin{theorem}[\protect \cite{Nadler}]
Let $\left( X,\rho \right) $ be a complete metric space and $T:X\rightarrow
CB(X)$ be a multivalued $\lambda $-contraction on $X$, that is, there exists 
$\lambda \in \lbrack 0,1)$ such that%
\begin{equation}
H\left( Tx,Ty\right) \leq \lambda \rho (x,y)  \label{MLC}
\end{equation}%
for all $x,y\in X$. Then, $T$ has a fixed point where $CB(X)$ is the class
of all closed and bounded subsets of $X$.
\end{theorem}

It is considered one of the most important classical results in the fixed
point theory, aside from the Banach contraction principle itself. Therefore,
this result continues to be expanded and studied by many researchers in
various contexts \cite{N1, N4, JPS, N2, N3}.

On the other hand, very recently Petrov introduce a new type of mappings
which can be characterized as mappings contracting perimeters of triangles.
The essential difference of the mappings contracting perimeters of triangles
from the contraction mapping in the sense of Banach is based on three points
of the space instead of two. Then, an interesting fixed point theorem was
obtained for such mappings. Before giving the result, let us remind an
important definition for us. Let $X$ be a nonempty set and $T:X\rightarrow X$
be a mapping. The point $x\in X$ is said to be periodic point of period $n$
if $x=T^{n}x$. The prime period of $x$ is the smallest positive $n$ for
which $x=T^{n}x~$[\cite{Prime}, p. 18].

\begin{theorem}[\protect \cite{Petrov}]
Let $\left( X,\rho \right) $ be a complete metric space with $\left \vert
X\right \vert \geq 3$ and $T:X\rightarrow X$ be a mapping such that it does
not possess periodic points of prime period $2$. If $T~$is a mapping
contracting perimeters of triangles on $X$, that is, there exists $\lambda
\in \lbrack 0,1)$ such that 
\begin{equation*}
\rho \left( Tx,Ty\right) +\rho \left( Ty,Tz\right) +\rho \left( Tx,Tz\right)
\leq \lambda \left( \rho (x,y)+\rho (y,z)+\rho (x,z)\right)
\end{equation*}%
for all three pairwise distinct points $x,y,z\in X$, then $T$ has a fixed
point. Moreover, the number of fixed points is at most two.
\end{theorem}

Also, it was showed that the class of contraction mapping is subclass of the
mappings contracting perimeters of triangles and so, the Banach contraction
principle was generalized in a different way from the results in literature.
In the same article, an example of obtaining an exact generalization of the
Banach contraction principle was given. Following, the applications of
Petrov's idea to well-known contractions in the fixed point literature were
given \cite{PC, PK}. For instance, although a Kannan type contraction $T$
with the constant $\lambda \in \left[ 0,\frac{1}{2}\right) $ on a metric
space $(X,\rho )~$is the mapping such that 
\begin{equation*}
\rho (Tx,Ty)\leq \lambda \left( \rho (x,Tx)+\rho (y,Ty)\right)
\end{equation*}%
for all $x,y\in X$, a generalized Kannan type mapping was defined as follows:

\begin{definition}[\protect \cite{PK}]
Let $\left( X,\rho \right) $ be a metric space with $\left \vert
X\right
\vert \geq 3$ and $T:X\rightarrow X$ be a mapping. If there exists $%
\lambda \in \left[ 0,\frac{2}{3}\right) ~$such that the inequality 
\begin{equation*}
\rho \left( Tx,Ty\right) +\rho \left( Ty,Tz\right) +\rho \left( Tx,Tz\right)
\leq \lambda \left( \rho (x,Tx)+\rho (y,Ty)+\rho (z,Tz)\right)
\end{equation*}%
is satisfied for all three pairwise distinct points $x,y,z\in X$ then, $T$
is said to be generalized Kannan type mapping.
\end{definition}

It was shown that the Kannan type contraction for the constant $\lambda \in %
\left[ 0,\frac{1}{3}\right) $ is a generalized Kannan type mapping. However,
in the same paper, an example was provided to demonstrate that the Kannan
type contraction for the constant $\lambda \in \left[ \frac{1}{3},\frac{1}{2}%
\right) $ may not be a generalized Kannan type mapping. Then, the following
theorem was presented.

\begin{theorem}[\protect \cite{PK}]
\label{K1}Let $\left( X,\rho \right) $ be a complete metric space with $%
\left \vert X\right \vert \geq 3$ and $T:X\rightarrow X$ be a mapping such
that it does not possess periodic points of prime period $2$. If $T~$is a
generalized Kannan type mapping on $X$, then $T$ has a fixed point.
Moreover, the number of fixed points is at most two.
\end{theorem}

The mapping given in Example 2.6 in \cite{PK} satisfies all the hypotheses
of Theorem \ref{K1}, and since it is not a Kannan type contraction, Theorem %
\ref{K1} is important.

In this paper, taking into account Petrov's new approach, we aim to obtain
fixed point results for multivalued mappings. In this context, we introduce
a new type of mapping called multivalued $\lambda $-contracting perimeters
of triangles. Then, we show that the class of multivalued $\lambda $%
-contractions construct an important subclass of these mappings with some
supporting examples. We also define the property of forming a triangle, thus
weakening the condition that the mapping $T$ must not have a point with
prime period $2$. Moreover, we prove a fixed point theorem for such mappings
and investigate the relation between the result and Nadler's fixed point
theorem. When analyzing the proof of this theorem, the construction of the
iteration sequence is not as classical as in the proof of fixed point
results for multivalued mappings. Finally, we provide an example to show
that our main result is crucial.

\section{Main Result}

We start this section with the definition of multivalued $\lambda $%
-contracting perimeters of triangles.

\begin{definition}
Let $\left( X,\rho \right) $ be a metric space with $\left \vert
X\right
\vert \geq 3$ and $T:X\rightarrow CB(X)$ be a multivalued mapping.
Then, it is said to be multivalued $\lambda $-contracting perimeters of
triangles on $X$ if there exists $\lambda \in \lbrack 0,1)$ such that 
\begin{equation}
H\left( Tx,Ty\right) +H\left( Ty,Tz\right) +H\left( Tx,Tz\right) \leq
\lambda \left( \rho (x,y)+\rho (y,z)+\rho (x,z)\right)  \label{MLCP}
\end{equation}%
for all three pairwise distinct points $x,y,z\in X$.
\end{definition}

\begin{remark}
Note that the points $x,y,z$ must be different. Otherwise, the definition
becomes the definition of multivalued $\lambda $-contraction mapping.
\end{remark}

Now, let $\left( X,\rho \right) $ be a metric space with $\left \vert
X\right \vert \geq 3$. Then, it is clear that every multivalued $\lambda $%
-contraction mapping on $X$ is a multivalued $\lambda $-contracting
perimeters of triangles on $X$. However, the converse may not be true. The
following example shows this fact.

\begin{example}
\label{Ex1}Let $X=\left \{ 0,1,2,3,4,5\right \} $ and define a function $%
\rho :X\times X\rightarrow 
\mathbb{R}
$ by 
\begin{equation*}
\rho (x,y)=\left \{ 
\begin{array}{ccc}
0 & , & x=y \\ 
2 & , & x,y\text{ are both even or odd} \\ 
1 & , & \text{otherwise}%
\end{array}%
\right. .
\end{equation*}%
Then, $\left( X,\rho \right) $ is a metric space. If we define a multivalued
mapping $T:X\rightarrow CB\left( X\right) $ by 
\begin{equation*}
Tx=\left \{ 
\begin{array}{ccc}
\left \{ 2,3\right \} & , & x\in \left \{ 0,1\right \} \\ 
\left \{ 4,5\right \} &  & x\in \left \{ 2,3\right \} \\ 
\left \{ 0,1\right \} & , & x\in \left \{ 4,5\right \}%
\end{array}%
\right. ,
\end{equation*}%
then it is not multivalued $\lambda $-contraction mapping on $X$. If we take 
$x=0$ and $y=3$, since 
\begin{equation*}
H(Tx,Ty)=1~\text{and }\rho (x,y)=1
\end{equation*}%
and so, the mapping $T$ is not multivalued $\lambda $-contraction on $X$.
However, for distinct $x,y,z\in X$ the inequality (MLCP) is satisfied. When
the all cases are analyzed, it can be seen that maksimum of the left side in
(\ref{MLCP}) is $3$ and the minimum of right side in (\ref{MLCP}) is $4$.
Hence, the mapping $T$ is a multivalued $\lambda $-contracting perimeters of
triangles on $X$ for all$~\frac{3}{4}$ $\leq \lambda <1.$
\end{example}

The following propositions are important to show when multivalued $\lambda $%
-contracting perimeters of triangles is multivalued $\lambda $-contraction
mapping.

\begin{proposition}
Every multivalued $\lambda $-contracting perimeters of triangles is
continuous.
\end{proposition}

\begin{proof}
Let $\left( X,\rho \right) $ be a metric space with $\left \vert
X\right
\vert \geq 3$ and $T:X\rightarrow CB(X)$ be a multivalued $\lambda $%
-contracting perimeters of triangles on $X$. If $x_{0}$ is a isolated point
in $X$, then $T$ is continuous at $x_{0}$. Now, let $x_{0}$ be an
accumulation point in $X$. Assume that $\left \{ x_{n}\right \} $ is an
arbitrary sequence in $X$ such that $x_{n}\rightarrow x_{0}$. Since $x_{0}$
is an accumulation point, there exists a sequence $\left \{ y_{n}\right \} ~$%
with $y_{n}\neq x_{n}$ and $y_{n}\neq x_{0}$ for all $n\in 
\mathbb{N}
$ such that $y_{n}\rightarrow x_{0}$ as $n\rightarrow \infty $. Using the
inequality (\ref{MLCP}) we have%
\begin{eqnarray*}
H\left( Tx_{0},Tx_{n}\right) &\leq &H\left( Tx_{0},Tx_{n}\right) +H\left(
Tx_{n},Ty_{n}\right) +H\left( Tx_{0},Ty_{n}\right) \\
&\leq &\lambda \left( \rho (x_{0},x_{n})+\rho (x_{n},y_{n})+\rho
(x_{0},y_{n})\right) \\
&\leq &2\lambda \left( \rho (x_{0},x_{n})+\rho (x_{0},y_{n})\right)
\end{eqnarray*}%
for all $n\in 
\mathbb{N}
.$Taking limit as $n\rightarrow \infty $ we get 
\begin{equation*}
\lim_{n\rightarrow \infty }H\left( Tx_{0},Tx_{n}\right) =0\text{. }
\end{equation*}%
So that, $T$ is continuous on $X$.
\end{proof}

\begin{proposition}
Let $\left( X,\rho \right) $ be a metric space with $\left \vert
X\right
\vert \geq 3$ and $T:X\rightarrow CB(X)$ be multivalued $\lambda $%
-contracting perimeters of triangles on $X$. If $x$ is a accumulation point
of $X$, then the inequality (\ref{MLC}) is satisfied for all $y\in X$.
\end{proposition}

\begin{proof}
Let $x$ be an accumulation point of $X$ and $y\in X$. If $y=x$, then the
inequality (\ref{MLC}) is satisfied. Now, let $y\neq x$. Since $x$ is an
accumulation point, there exists a sequence $\left \{ x_{n}\right \} ~$with $%
x_{n}\neq y$ and $x_{n}\neq x$ for all $n\in 
\mathbb{N}
$ such that $x_{n}\rightarrow x$ as $n\rightarrow \infty $. Using the
inequality (\ref{MLCP}) we have%
\begin{equation}
H\left( Tx,Ty\right) +H\left( Ty,Tx_{n}\right) +H\left( Tx,Tx_{n}\right)
\leq \lambda \left( \rho (x,y)+\rho (y,x_{n})+\rho (x,x_{n})\right)
\label{E999}
\end{equation}%
for all $n\in 
\mathbb{N}
$. Since $H\left( Tx,Tx_{n}\right) \rightarrow 0$ and $H\left(
Ty,Tx_{n}\right) \rightarrow H\left( Ty,Tx\right) $ because of continuity of 
$T$, from $\left( \ref{E999}\right) ~$we have 
\begin{equation*}
H\left( Tx,Ty\right) +H\left( Ty,Tx\right) \leq \lambda \left( \rho
(x,y)+\rho (x,y)\right) ,
\end{equation*}%
and so%
\begin{equation*}
H\left( Tx,Ty\right) \leq \lambda \rho (x,y).
\end{equation*}
\end{proof}

\begin{corollary}
Let $\left( X,\rho \right) $ be a metric space with $\left \vert
X\right
\vert \geq 3$ and $T:X\rightarrow CB(X)$ be multivalued $\lambda $%
-contracting perimeters of triangles on $X$. If there is no isolated point
in $X$, then $T $ is a multivalued $\lambda $-contraction mapping.
\end{corollary}

Let $X$ be a nonempty set and $T:X\rightarrow P\left( X\right) $ be a
multivalued mapping. The point $x\in X$ is said to be periodic point of
period $n$ if $x\in T^{n}x$. The prime period of $x$ is the smallest
positive $n$ for which $x\in T^{n}x$.

Now, we introduce a new property which is important for our main result.

\begin{definition}
Let $\left( X,\rho \right) $ be a metric space and $T:X\rightarrow CB(X)$ be
a multivalued mapping. If for all $x\in X$, $y\in Tx$ with $x\neq y$ and for
all $\alpha >0$, there exists $z\in Ty$ with $H(Tx,Tz)>0$ such that%
\begin{equation*}
\rho (y,z)\leq H(Tx,Ty)+\alpha ,
\end{equation*}%
then the mapping $T$ is called to have the property of forming a triangle.
\end{definition}

Lemma \ref{L1} shows the property of forming a triangle for a multivalued
mapping $T$ weaken the condition that the mapping $T$ does not possess
periodic points of prime $2$.

\begin{lemma}
\label{L1}Let $\left( X,\rho \right) $ be a metric space and $T:X\rightarrow
CB(X)$ be a multivalued mapping without fixed point. If the mapping $T$ does
not possess periodic points of prime $2$, then $T$ has the property of
forming a triangle.
\end{lemma}

\begin{proof}
Let $\left( X,\rho \right) $ be a metric space and $T:X\rightarrow CB(X)$ be
a multivalued mapping without fixed point. Assume that the mapping $T$ does
not possess periodic points of prime period $2$. We want to show that $T$
has the property of forming a triangle. Let's take arbitrary $x\in X$, $y\in
Tx$ with $x\neq y$ and $\alpha >0$. In this case, there exists $z\in Ty$
such that%
\begin{equation*}
\rho (y,z)\leq H(Tx,Ty)+\alpha .
\end{equation*}%
To finish the proof, it is enough to demonstrate $H(Tx,Tz)>0$. Assume the
contrary. If $H(Tx,Tz)=0$, we get $Tx=Tz$. Then, we have $y\in Tz.$Hence we
get $z\in Ty\subseteq T^{2}z,$ and so $z\in T^{2}z~$which contradicts the
assumption that $T$ does not possess periodic points of prime period $2$,
since $z$ is not a fixed point of $T$. Hence, $T$ has the property of
forming a triangle.
\end{proof}

The following example shows that reverse of the Lemma \ref{L1} may not be
true.

\begin{example}
Let $X=\left \{ 0,1,2,3,4,5,6\right \} $. Consider the metric $\rho $ in
Example \ref{Ex1}. Now, we define a multivalued mapping $T:X\rightarrow
CB(X) $ by 
\begin{equation*}
Tx=\left \{ 
\begin{array}{ccc}
\left \{ 1,2,3,4\right \} & , & x=0 \\ 
\left \{ 2,3,4,5\right \} & , & x=1 \\ 
\left \{ 3,4,5,6\right \} & , & x=2 \\ 
\left \{ 4,5,6,0\right \} & , & x=3 \\ 
\left \{ 5,6,0,1\right \} & , & x=4 \\ 
\left \{ 6,0,1,2\right \} & , & x=5 \\ 
\left \{ 0,1,2,3\right \} & , & x=6%
\end{array}%
\right. .
\end{equation*}%
Then, the mapping $T$ has the property of forming a triangle. Indeed, let $%
x\in X$ and $y\in Tx$, we have $H\left( Tx,Ty\right) =1$. Also, we can find $%
z\in Ty$ such that $H(Tx,Tz)>0$ and $\rho (y,z)=1$ since $Ty$ includes both
even and odd numbers. In this case, we have 
\begin{equation*}
\rho (y,z)\leq H(Tx,Ty)+\alpha
\end{equation*}%
for $\alpha >0$. On the other hand, $T$ has a periodic point of prime period 
$2$.
\end{example}

Now, we give our main result.

\begin{theorem}
\label{Theo1}Let $\left( X,\rho \right) $ be a complete metric space with $%
\left \vert X\right \vert \geq 3$ and $T:X\rightarrow CB(X)$ be a
multivalued $\lambda $-contracting perimeters of triangles on $X$ with $%
\lambda \in \left[ 0,\frac{1}{2}\right) $. If the mapping $T~$has the
property of forming a triangle, then $T$ has a fixed point.
\end{theorem}

\begin{proof}
Except for the obvious situation, we can assume that $\lambda >0$ in the
proof. Let $x_{0}\in X$ and $x_{1}\in Tx_{0}$. If $x_{0}=x_{1}$, then the
proof is finished. Hence, we assume that $x_{0}\neq x_{1}$. Then, since $T$
has the property of forming a triangle, there exists $x_{2}\in Tx_{1}$ with $%
H(Tx_{0},Tx_{2})>0$, and so $x_{0}\neq x_{2}$ such that 
\begin{equation}
\rho (x_{1},x_{2})\leq H(Tx_{0},Tx_{1})+\lambda .  \label{E1}
\end{equation}%
Now, we can consider the condition $x_{1}\neq x_{2}$ (otherwise the proof is
finished). Using the property of forming a triangle of $T$ there exists $%
x_{3}\in Tx_{2}~$with $H(Tx_{1},Tx_{3})>0$, and so $x_{1}\neq x_{3}$ such
that 
\begin{equation}
\rho (x_{2},x_{3})\leq H(Tx_{1},Tx_{2})+\min \left \{ \lambda
^{2},H(Tx_{0},Tx_{2})\right \} .  \label{E2}
\end{equation}

Continuing this process, we can construct a sequence $\left \{
x_{n}\right
\} $ whose three consecutive terms of the sequence are
different from each other with $H\left( Tx_{n-2},Tx_{n}\right) >0$ such that%
\begin{equation*}
\rho (x_{1},x_{2})\leq H(Tx_{0},Tx_{1})+\lambda
\end{equation*}%
and 
\begin{equation}
\rho (x_{n},x_{n+1})\leq H(Tx_{n-1},Tx_{n})+\min \left \{ \lambda
^{n},H(Tx_{n-2},Tx_{n})\right \}  \label{E4}
\end{equation}%
for all $n\geq 2$. Using inequalities $\left( \ref{E1}\right) $ and $\left( %
\ref{E2}\right) $, we have 
\begin{eqnarray}
\rho (x_{1},x_{2})+\rho (x_{2},x_{3}) &\leq &H(Tx_{0},Tx_{1})+\lambda
+H(Tx_{1},Tx_{2})  \notag \\
&&+\min \left \{ \lambda ^{2},H(Tx_{0},Tx_{2})\right \}  \notag \\
&\leq &H(Tx_{0},Tx_{1})+H(Tx_{1},Tx_{2})+H(Tx_{0},Tx_{2})+\lambda  \notag \\
&\leq &\lambda \left( \rho (x_{0},x_{1})+\rho (x_{1},x_{2})+\rho
(x_{0},x_{2})\right) +\lambda  \notag \\
&\leq &2\lambda \left( \rho (x_{0},x_{1})+\rho (x_{1},x_{2})\right) +\lambda
.  \label{E6}
\end{eqnarray}%
In a similar way, from the inequality $\left( \ref{E4}\right) $, we get%
\begin{eqnarray}
\rho (x_{2},x_{3})+\rho (x_{3},x_{4}) &\leq &H(Tx_{1},Tx_{2})+\min \left \{
\lambda ^{2},H(Tx_{0},Tx_{2})\right \}  \notag \\
&&+H(Tx_{2},Tx_{3})+\min \left \{ \lambda ^{3},H(Tx_{1},Tx_{3})\right \} 
\notag \\
&\leq &H(Tx_{1},Tx_{2})+H(Tx_{2},Tx_{3})+H(Tx_{1},Tx_{3})+\lambda ^{2} 
\notag \\
&\leq &\lambda \left( \rho (x_{1},x_{2})+\rho (x_{2},x_{3})+\rho
(x_{1},x_{3})\right) +\lambda ^{2}  \notag \\
&\leq &2\lambda \left( \rho (x_{1},x_{2})+\rho (x_{2},x_{3})\right) +\lambda
^{2}.  \label{E7}
\end{eqnarray}%
Hence, we get 
\begin{equation*}
\rho (x_{n},x_{n+1})+\rho (x_{n+1},x_{n+2})\leq 2\lambda \left( \rho
(x_{n-1},x_{n})+\rho (x_{n},x_{n+1})\right) +\lambda ^{n}
\end{equation*}%
for all $n\geq 1$. Hence, we conclude that 
\begin{eqnarray*}
\rho (x_{n},x_{n+1})+\rho (x_{n+1},x_{n+2}) &\leq &2\lambda \left( \rho
(x_{n-1},x_{n})+\rho (x_{n},x_{n+1})\right) +\lambda ^{n} \\
&\leq &2\lambda \left \{ 2\lambda \left( \rho (x_{n-2},x_{n-1})+\rho
(x_{n-1},x_{n})\right) +\lambda ^{n-1}\right \} \\
&&+2\lambda ^{n} \\
&=&\left( 2\lambda \right) ^{2}\left( \rho (x_{n-2},x_{n-1})+\rho
(x_{n-1},x_{n})\right) +4\lambda ^{n} \\
&\leq &\left( 2\lambda \right) ^{2}\left \{ 2\lambda \left( \rho
(x_{n-3},x_{n-2})+\rho (x_{n-2},x_{n-1})\right) +\lambda ^{n-2}\right \} \\
&&+8\lambda ^{n} \\
&=&\left( 2\lambda \right) ^{3}\left( \rho (x_{n-3},x_{n-2})+\rho
(x_{n-2},x_{n-1})\right) +8\lambda ^{n} \\
&\leq &\left( 2\lambda \right) ^{3}\left \{ 2\lambda \left( \rho
(x_{n-4},x_{n-3})+\rho (x_{n-3},x_{n-2})\right) +\lambda ^{n-3}\right \} \\
&&+8\lambda ^{n} \\
&=&\left( 2\lambda \right) ^{4}\left( \rho (x_{n-4},x_{n-3})+\rho
(x_{n-3},x_{n-2})\right) +16\lambda ^{n} \\
&&\vdots \\
&\leq &\left( 2\lambda \right) ^{n}\left( \rho (x_{0},x_{1})+\rho
(x_{1},x_{2})\right) +2^{n}\lambda ^{n}
\end{eqnarray*}%
for all $n\geq 1$. Hence, we have 
\begin{equation*}
\sum \limits_{n=1}^{\infty }\rho (x_{n},x_{n+1})\leq \sum
\limits_{n=1}^{\infty }\left( 2\lambda \right) ^{n}\left( \rho
(x_{0},x_{1})+\rho (x_{1},x_{2})+1\right) <\infty .
\end{equation*}

So that, $\left \{ x_{n}\right \} $ is a Cauchy sequence. Because of the
completeness of the space $(X,\rho )$, there exist a point $x^{\ast }$ in $X$
such that $x_{n}\rightarrow x^{\ast }$. Let define a set 
\begin{equation*}
A=\{n_{i}:x_{n_{i}}=x^{\ast }\text{ for }i\in 
\mathbb{N}
\}.
\end{equation*}%
First, we claim that $A$ is finite set. Assume the contrary. Then, since $%
x_{n}$ is not a fixed point of the mapping $T$ for all $n\in 
\mathbb{N}
$, we have 
\begin{equation*}
0<\rho \left( x_{n_{i}},Tx_{n_{i}}\right) =\rho (x^{\ast },Tx^{\ast })
\end{equation*}%
for all $n_{i}\in A$. Also, for all $n_{i}\in A$ we have 
\begin{equation*}
\rho (x^{\ast },Tx^{\ast })=\rho \left( x_{n_{i}},Tx_{n_{i}}\right) \leq
\rho \left( x_{n_{i}},x_{n_{i}+1}\right) \rightarrow 0
\end{equation*}%
as $n_{i}\rightarrow \infty $ which is a contradiction. Hence, $A$ is finite
set. Now, let 
\begin{equation*}
n_{0}=\left \{ 
\begin{array}{ccc}
0 & , & A=\emptyset \\ 
&  &  \\ 
\max A & , & A\neq \emptyset%
\end{array}%
\right. .
\end{equation*}%
Then, we have $x_{n}\neq x^{\ast }$ for all $n>n_{0}$. Hence, we get 
\begin{eqnarray*}
\rho \left( x_{n+1},Tx^{\ast }\right) &\leq &H(Tx_{n},Tx^{\ast
})+H(Tx_{n},Tx_{n+1})+H(Tx_{n+1},Tx^{\ast }) \\
&\leq &\lambda (\rho (x_{n},x^{\ast })+\rho (x_{n},x_{n+1})+\rho
(x_{n+1},x^{\ast }))
\end{eqnarray*}%
for all $n>n_{0}$. Taking limit $n\rightarrow \infty $ in the last
inequality we have 
\begin{equation*}
\rho \left( x^{\ast },Tx^{\ast }\right) =0.
\end{equation*}%
So that, we get $x^{\ast }\in \overline{Tx^{\ast }}=Tx^{\ast }$. That is, $%
x^{\ast }$ is a fixed point of $T$.
\end{proof}

Using the Lemma \ref{L1} we can obtain the following result from Theorem \ref%
{Theo1}.

\begin{corollary}
Let $\left( X,\rho \right) $ be a complete metric space with $\left \vert
X\right \vert \geq 3$ and $T:X\rightarrow CB(X)$ be a multivalued $\lambda $%
-contracting perimeters of triangles on $X$ with $\lambda \in \left[ 0,\frac{%
1}{2}\right) $. If the mapping $T~$does not possess periodic points of prime 
$2$, then $T$ has a fixed point.
\end{corollary}

\begin{proof}
Assume that $T~$does not possess periodic points of prime $2$. We want to
show that $T$ has a fixed point. Assume the contrary. In this case, from
Lemma \ref{L1}, we conclude $T~$has the property of forming a triangle.
Hence, all hypotheses of Theorem \ref{Theo1} hold, and so $T$ must have a
fixed point. This is a contradiction.
\end{proof}

The following example is significant in that it demonstrates that Nadler's
result is not applicable, while our main theorem could still be applied.

\begin{example}
Consider $X=\left \{ 0,1,4,5,6\right \} $ with the usual metric $\rho $ and
let $T:X\rightarrow CB(X)$ be a multivalued mapping defined by%
\begin{equation*}
Tx=\left \{ 
\begin{array}{ccc}
\left \{ 0\right \} & , & x=0 \\ 
\left \{ 1\right \} & , & x=1 \\ 
\left \{ 0,1\right \} & , & x=4,5,6%
\end{array}%
\right. .
\end{equation*}%
Then, $T$ is not multivalued $\lambda $-contraction on $X$, since $%
H(T0,T1)=\rho (0,1)$. On the other hand, by the routine calculation we can
see that $T$ is a multivalued $\lambda $-contracting perimeters of triangles
on $X$ with $\lambda =\frac{3}{8}$. Also, $T~$does not possess periodic
points of prime $2$, and so it has the property of forming a triangle.
Hence, since all hypotheses of Theorem \ref{Theo1} are satisfied, $T$ has a
fixed point.
\end{example}

\section{Conclusion}

Our main result established that if $\lambda <\frac{1}{2}$, then any
multivalued $\lambda $-contracting perimeters of triangles mapping in a
complete metric space necessarily possesses a fixed point. We also emphasize
the significance of the restriction by showing, through a constructed
example, that if $\lambda \geq \frac{3}{4}$, such maps may not have a fixed
point. However, the existence of fixed points for multivalued $\lambda $%
-contracting perimeters of triangles mappings remains an open problem in the
case $\lambda \in \left[ \frac{1}{2},\frac{3}{4}\right) $. Determining
whether a fixed point always exists under this condition is an interesting
question for future research.

\begin{center}
\textbf{Author contributions}
\end{center}

H. S., M. A. and I.A. equally contributed to the development of the paper.

\begin{center}
\textbf{Funding}
\end{center}

There is no funding for the research article.

\begin{center}
\textbf{Data availability}
\end{center}

No datasets were generated or analyzed during the current study.

\begin{center}
\textbf{Conflict of interest}
\end{center}

The authors have no conflict of interest to declare that are relevant to the
content of this article.

\end{document}